\documentclass[10pt,reqno]{amsart}

\usepackage{amssymb}
\usepackage{amscd}
\usepackage{amsfonts}
\usepackage{setspace}
\usepackage{version}
\usepackage{enumerate}
\usepackage{hyperref}

\newtheorem{theorem}{Theorem}[section]
\newtheorem{lemma}[theorem]{Lemma}
\newtheorem{proposition}[theorem]{Proposition}
\newtheorem{corollary}[theorem]{Corollary}

\theoremstyle{definition}
\newtheorem{question}[theorem]{Question}
\newtheorem{conjecture}[theorem]{Conjecture}

\newtheorem{remark}[theorem]{Remark}

\renewcommand{\(}{\left(}
\renewcommand{\)}{\right)}

\newcommand{\trace}[0]{\operatorname{Tr}}

\def\R{\mathbb{R}}
\def\C{\mathbb{C}}
\def\Z{\mathbb{Z}}
\def\E{\mathbb{E}}

\def\N{\mathbb{N}}

\def\cO{\mathcal{O}}
\def\eps{\varepsilon}

\newcommand{\bigE}{\mbox{\LARGE$\E$}}

\renewcommand{\hat}{\widehat}
\newcommand{\sump}{\sideset{}{'}\sum}

\numberwithin{equation}{section}

\begin{document}

\title{\mbox{Additive triples of bijections}, or \mbox{the toroidal semiqueens problem}}



\author{Sean Eberhard}
\address{Sean Eberhard, Mathematical Institute, University of Oxford}
\email{eberhard@maths.ox.ac.uk}

\author{Freddie Manners}
\address{Freddie Manners, Mathematical Institute, University of Oxford}
\email{manners@maths.ox.ac.uk}

\author{Rudi Mrazovi\'c}
\address{Rudi Mrazovi\'c, Mathematical Institute, University of Oxford}
\email{mrazovic@maths.ox.ac.uk}

\onehalfspace


\begin{abstract}
We prove an asymptotic for the number of additive triples of bijections $\{1,\dots,n\}\to\Z/n\Z$, that is, the number of pairs of bijections $\pi_1,\pi_2\colon \{1,\dots,n\}\to\Z/n\Z$ such that the pointwise sum $\pi_1+\pi_2$ is also a bijection. This problem is equivalent to counting the number of orthomorphisms or complete mappings of $\Z/n\Z$, to counting the number of arrangements of $n$ mutually nonattacking semiqueens on an $n\times n$ toroidal chessboard, and to counting the number of transversals in a cyclic Latin square. The method of proof is a version of the Hardy--Littlewood circle method from analytic number theory, adapted to the group $(\Z/n\Z)^n$. 
\end{abstract}

\maketitle
\tableofcontents

\section{Introduction}

Let $S$ be the set of bijections $\{1,\dots,n\}\to\Z/n\Z$, thought of as a subset of the group $(\Z/n\Z)^n$. We are interested in counting the number $s_n$ of additive triples in $S$, that is, the number of pairs of bijections $\pi_1,\pi_2\colon \{1,\dots,n\}\to\Z/n\Z$ such that the pointwise sum $\pi_1+\pi_2$ is also a bijection.

The number $s_n$ has been studied somewhat extensively, but under a different guise. Since $S$ is invariant under precomposition with permutations of $\{1,\dots,n\}$ it is easy to see by arbitrarily identifying $\{1,\dots,n\}$ with $\Z/n\Z$ that $s_n/n!$ is number of permutations $\pi$ of $\Z/n\Z$ such that $x\mapsto \pi(x)-x$ is also a permutation. Such maps are called orthomorphisms or complete mappings, and afford the following fun interpretation. Define a semiqueen to be a chess piece which can move any distance horizontally, vertically, or diagonally in the northeast-southwest direction. Then orthomorphisms represent ways of arranging $n$ mutually nonattacking semiqueens on an $n\times n$ toroidal chessboard. Thus $s_n/n!$ is the number of such arrangements.

In another guise this problem is that of counting the number of transversals of a cyclic Latin square. Here a transversal of an $n\times n$ Latin square is a set of $n$ squares with no two sharing the same row, column, or symbol, and the cyclic Latin square is the Latin square with $(i,j)$ entry given by $i+j\pmod{n}$. Then as above the number of such transversals is $s_n/n!$.

If $n$ is even then $s_n=0$. Indeed, in this case for any bijection $\pi\colon \{1,\dots,n\}\to \Z/n\Z$ we have $\sum_{i=1}^n \pi(i) = n/2$, so the sum of two bijections is never again a bijection. If $n$ is odd then it is easy to see that $s_n>0$, but estimating $s_n$ is not easy.

In 1991, Vardi~\cite{vardi} made a conjecture equivalent to the following.

\begin{conjecture}\label{vardi}
There are constants $c_1>0$ and $c_2<1$ such that for any large enough odd number $n$ we have
\[
	c_1^nn!^2 \leq s_n \leq c_2^nn!^2.
\]
\end{conjecture}

The upper bound in this conjecture is known, and various authors have made incremental improvements to the constant $c_2$. Cooper and Kovalenko~\cite{cooperkovalenko} showed that $c_2 = e^{-0.08854}$ is acceptable, and this was later improved by Kovalenko~\cite{kovalenko} to $c_2 = 1/\sqrt{2}$ and by McKay, McLeod and Wanless~\cite{mckay} to $c_2 = 0.614$. More recently, Taranenko~\cite{taranenko} proved that one can take $c_2 = 1/e + o(1)$. Glebov and Luria~\cite{glebovluria} proved the same bound using a somewhat simpler method based on entropy.

There has been much less progress on the lower bound, but nontrivial lower bounds for $s_n$ have been proved under various arithmetic assumptions about $n$. For example, Cooper~\cite{cooperlower} proved that $s_n \geq e^{\frac12\sqrt{n}\log n}n!$ under the hypothesis that $n$ has a divisor of size roughly $\sqrt{n}$, while if, say, $n$ is prime then a lower bound of Rivin, Vardi, and Zimmermann~\cite{rivinvardizimmermann} for the torodial queens problem gives $s_n \geq 2^{\sqrt{(n-1)/2}}n!$. These lower bounds were recently superseded by work of Cavenagh and Wanless~\cite{cavenaghwanless}, which showed that $s_n\geq 3.246^n n!$ for all odd $n$. However, this lower bound is still a long way from the one in Conjecture~\ref{vardi}.

Some researchers have also tried investigating the growth rate of $s_n$ numerically: see Cooper, Gilchrist, Kovalenko, and Novakovi\'{c}~\cite{coopersimulations} and Kuznetsov~\cite{kuznetsovone, kuznetsovtwo, kuznetsovthree}. Based on this numerical evidence, Wanless~\cite{wanless} conjectured that we can take $c_1,c_2 = 1/e+o(1)$; specifically he conjectured that
\[
  \lim_{n\to\infty} \frac1n \log(s_n/n!^2) = -1.
\]
Finally, we should mention that $s_n/(n\cdot n!)$ is sequence A003111 in \cite{oeis}.


In this paper we prove Vardi's conjecture with the optimal values $c_1,c_2 = 1/e + o(1)$, thus also confirming Wanless's conjecture. In fact our estimate is much more precise -- we compute $s_n$ up to a factor of $1+o(1)$.

\begin{theorem}\label{supercooltheorembetterthanfermatslast}
Let $n$ be an odd integer. Then
\[
	s_n = (e^{-1/2}+o(1))n!^3/n^{n-1}.
\]
\end{theorem}

Perhaps the most surprising feature of this estimate is the appearance of the constant $e^{-1/2}$. This constant arises as the sum of the singular series in an argument resembling the circle method from analytic number theory, but it can be rationalized heuristically as follows. If $\pi_1,\pi_2\colon \{1,\dots,n\}\to\Z/n\Z$ are random bijections then the sum $f=\pi_1+\pi_2$ is something like a random function subject to $\sum_{x\in\Z/n\Z} f(x)=0$, and so we might guess that the probability $\pi_1+\pi_2$ is a bijection is about $n\cdot n!/n^n$. However if we fix two elements $x,y\in\Z/n\Z$ then the difference
\[
	f(x)-f(y) = (\pi_1(x) - \pi_1(y)) + (\pi_2(x) - \pi_2(y))
\]
is the sum of two uniformly random nonzero elements, so $f(x)=f(y)$ with probability $1/(n-1)$, not $1/n$. Thus the probability that $f(x)\neq f(y)$ is smaller than we previously suggested by a factor of
\[
	\frac{1-1/(n-1)}{1-1/n} = 1 - 1/n^2 + O(1/n^3).
\]
There are a total of $n^2/2 + O(n)$ pairs $x,y$, so we might guess that the probability $\pi_1+\pi_2$ is a bijection is more like
\[
	\(1-1/n^2\)^{n^2/2} n\cdot n!/n^n \approx e^{-1/2} n!/n^{n-1}.
\]

Our proof applies with only notational modifications when $\Z/n\Z$ is replaced by any abelian group $G$ of odd order. To be precise, given a finite abelian group $G$, let $s(G)$ be the number of pairs of bijections $\pi_1,\pi_2:\{1,\dots,|G|\}\to G$ such that $\pi_1+\pi_2$ is also a bijection. Then by the same method which proves Theorem~\ref{supercooltheorembetterthanfermatslast}, we have the following theorem.
\begin{theorem}
Let $G$ be a finite abelian group of odd order $n$. Then
\[
  s(G) = (e^{-1/2}+o(1)) n!^3/n^{n-1}.
\]
\end{theorem}
There are additional complications however when $G$ is allowed to have even order, say when $G = (\Z/2\Z)^d$, and we do not know whether the same method can be made to work in this case.

Like the toroidal semiqueens problem, the toroidal queens problem -- the problem of counting the number of ways of arranging $n$ mutually nonattacking queens on an $n \times n$ toroidal chessboard -- is equivalent to counting the number of solutions to a particular linear system in the set of bijections, namely the system
\begin{align*}
\pi_1+\pi_2&=\pi_3,\\
\pi_1-\pi_2&=\pi_4.
\end{align*}
Linear systems of complexity $2$ like this one are known not to be controlled by Fourier analysis alone, but there is some hope that one could use the higher-order theory pioneered by Gowers. This is an interesting avenue which has not yet been fully explored.

\emph{Notation.} Although we have already implicitly introduced most of the notation and conventions, we include them here for the reader's convenience. We use the standard $O$-notation. To be concrete, for functions $f,g\colon \N\to \R$ we write $f(n)=O(g(n))$ if $|f(n)|\leq Cg(n)$ for some constant $C>0$, or $f(n)\leq O_X(g(n))$ if $C$ is allowed to depend on the parameter $X$. We write $f(n)=o(g(n))$ if $f(n)/g(n)\to 0$ as $n$ tends to infinity. For a finite set $S$ and function $f$, we will denote the average of $f$ on the set $S$ by $\E_{s\in S}f(s)$. We use a primed sum $\sum_{s_1,\dots,s_k\in S}'$ to denote the sum over all $k$-tuples of distinct elements $s_1,\dots, s_k\in S$. All the groups we will work with will be finite, and we equip each of them with the either uniform probability or counting measure, depending on whether we are working on the physical or the frequency side, respectively. Of course, this convention is respected in our definitions of convolution, inner product and $L^2$-norm. Finally, we will also use the standard notation $e(x) = e^{2\pi i x}$.

\emph{Acknowledgements.} We would like to thank Ben Green for communicating the problem of estimating $s_n$, Robin Pemantle for a discussion in connection with the method used to prove Theorem~\ref{SRH}, and Fernando Shao for useful conversations.

\section{Outline of the proof}\label{idiotsguide}

Our approach to proving Theorem~\ref{supercooltheorembetterthanfermatslast} is Fourier-analytic. Write $G = \Z/n\Z$ for $n$ an odd integer, and write $S \subset G^n$ for the set of bijections $\{1,\dots,n\} \to G$.  Our goal is to compute the quantity
\[
  \E_{x,y \in G^n} 1_S(x) 1_S(y) 1_S(x+y) = \langle 1_S \ast 1_S, 1_S \rangle.
\]
A standard application of Parseval's identity shows that this quantity can be expressed in terms of the Fourier transform of $1_S$ as
\begin{equation} \label{eqn:summing-triples}
  \sum_{\chi \in \widehat{G}^n} |\widehat{1_S}(\chi)|^2 \widehat{1_S}(\chi).
\end{equation}
Here we identify $\widehat{G}$ with $\frac1n\Z/\Z$ in the usual way, so that for $\chi = (r_1, \dots, r_n) \in (\frac1n\Z/\Z)^n$ we have explicitly
\begin{equation}
  \label{eqn:ft-defn}
  \widehat{1_S}(r_1,\dots,r_n) = \frac1{n^n} \sump_{x_1,\dots,x_n} e(-r_1x_1-\cdots-r_nx_n),
\end{equation}
where, as mentioned in the previous section, we use $\sum'$ to denote the sum over distinct $x_1, \dots, x_n \in G$, i.e., the sum over $S$. In fact it is clear that $\widehat{1_S}$ is real-valued, since $S = -S$, and hence we may drop the absolute value signs in~\eqref{eqn:summing-triples}.

The form of the proof can then be viewed as an analogue of the Hardy--Littlewood circle method in analytic number theory, adapted to the group $G^n$ rather than $\Z$. We borrow some nomenclature from the classical setting.
\begin{itemize}
  \item There are a small number of characters $\chi \in \widehat{G}^n$, namely the characters $\chi=(r_1,\dots,r_n)$ for which almost all of the $r_i$ are equal, that make a substantial contribution to the sum~\eqref{eqn:summing-triples}. We call the totality of these $\chi$ the \emph{major arcs}. We compute explicitly the contribution of these $\chi$ to \eqref{eqn:summing-triples}, up to small errors: this is an analogue of the singular series, which accounts for the main term in Theorem~\ref{supercooltheorembetterthanfermatslast}, including the bizarre constant $e^{-1/2}$. For all this see Section~\ref{sec:major-arcs}.
  \item We bound the contribution from all other characters using the triangle inequality; that is, we obtain an upper bound for
    \[
      \sum_{\text{all other}~\chi} |\widehat{1_S}(\chi)|^3
    \]
that is smaller than the main term. We call such $\chi$ collectively the \emph{minor arcs}.
\end{itemize}

A large part of this paper is therefore devoted to obtaining good bounds for Fourier coefficients $\widehat{1_S}(\chi)$, either pointwise or on average in a suitable sense. In fact we will need to combine several different arguments which are effective in different regimes.

\subsection{Preliminaries on \texorpdfstring{$\widehat{1_S}$}{1\_S hat}}

In order to describe how the characters $\widehat{G}^n$ are divided up into different pieces in which different approaches will be effective, we will need some preliminary remarks.

First, note that the function $\widehat{1_S}(r_1, \dots, r_n)$ is invariant under permutation of the $r_i$:
\[
  \widehat{1_S}(r_1, \dots, r_n) = \widehat{1_S}(r_{\sigma(1)}, \dots, r_{\sigma(n)})
\]
for any bijection $\sigma \colon \{1,\dots,n\} \to \{1,\dots,n\}$.  This is immediate from the definition~\eqref{eqn:ft-defn}. Hence it makes sense to specify a Fourier coefficient of interest in the form $\widehat{1_S}(r_1^{a_1}, \dots, r_k^{a_k})$, where $r_i \in \hat{G}$ are distinct, the $a_i$ are positive integers such that $\sum a_i = n$, and the notation $r^a$ means $r$ repeated $a$ times.

Second, observe that $\hat{1_S}(r_1, \dots, r_n) = 0$ unless $\sum r_i = 0$.  This follows from the fact that $S$ is invariant under global shifts $(x_i) \mapsto (x_i + t)$, so one can compute
\begin{align*}
	\hat{1_S}(r_1,\dots,r_n)
	&= \frac1{n^n} \sump_{x_1,\dots,x_n} e(-r_1x_1-\cdots-r_nx_n)\\
	&= \frac1{n^n} \sump_{x_1,\dots,x_n} e(-r_1(x_1+t)-\cdots-r_n(x_n+t))\\
  &= \hat{1_S}(r_1,\dots,r_n) e(-(r_1+\cdots+r_n) t).
\end{align*}

Dually, note that $\widehat{1_S}$ is invariant under global shifts. This follows from the fact that $\sum x_i = 0$ for any $(x_i) \in S$.  Hence,
\begin{align*}
	\hat{1_S}(r_1,\dots,r_n)
	&= \frac1{n^n} \sump_{x_1,\dots,x_n} e(-r_1x_1-\cdots-r_nx_n)\\
  &= \frac1{n^n} \sump_{x_1,\dots,x_n} e(-(r_1+t)x_1-\cdots-(r_n+t)x_n)\\
	&= \hat{1_S}(r_1+t,\dots,r_n+t).
\end{align*}

Finally, note the trivial bound
  \[
    |\widehat{1_S}(\chi)| \le \widehat{1_S}(0) = \frac{n!}{n^n}.
  \]

\subsection{Entropy ranges for minor arcs}

We now explain a straightforward but still fairly powerful bound on $|\widehat{1_S}(\chi)|$ that follows directly from these elementary considerations.

\begin{proposition}[``Entropy bound'']  \label{prop:entropy}
  We have
  \[
    |\widehat{1_S}(r_1^{a_1}, \dots, r_k^{a_k})| \le \binom{n}{a_1, \dots, a_k}^{-1/2} \left(\frac{n!}{n^n}\right)^{1/2}.
  \]
\end{proposition}
\begin{proof}
  Let $\cO \subset \widehat{G}^n$ denote the set of characters obtained by permuting the elements of $(r_1^{a_1}, \dots, r_k^{a_k})$.  Hence, $|\cO| = \binom{n}{a_1, \dots, a_k}$ and by permutation invariance, $\widehat{1_S}$ takes a constant value on $\cO$. Thus by Parseval,
  \begin{align*}
    \frac{n!}{n^n}
    &= \sum_{\chi \in \widehat{G}^n} |\widehat{1_S}(\chi)|^2 \\
    &\ge \sum_{\chi \in \cO} |\widehat{1_S}(\chi)|^2 \\
    &= |\cO| |\widehat{1_S}(r_1^{a_1}, \dots, r_k^{a_k})|^2,
  \end{align*}
  and the result follows.
\end{proof}

The term ``entropy bound'' refers to the fact that the quantity $H(\chi) = \frac1{n} \log \binom{n}{a_1, \dots, a_k}$ is roughly the entropy $\sum_{i=1}^k (a_i/n) \log (n/a_i) \in [0, \log n]$ of a random variable taking the value $r_i$ with probability $a_i/n$.  In a slight abuse of nomenclature, we refer to this first quantity $H(\chi)$ as the \emph{entropy of $\chi$}.  In this language, the bound of Proposition \ref{prop:entropy} is precisely $\exp(- H n / 2) (n!/n^n)^{1/2}$ or roughly $\exp(- Hn/2-n/2)$. 

This is already sufficient to control the contribution to~\eqref{eqn:summing-triples} for most characters $\chi$.

\begin{corollary}
  \label{cor:crappy-minor-arcs}
  We have
  \[
    \sum_{\chi \colon H(\chi) \ge R} |\hat{1_S}(\chi)|^3 \le \exp((3 - R)n/2) \(\frac{n!}{n^n}\)^3.
  \]
\end{corollary}
\begin{proof}
  By Parseval and Proposition~\ref{prop:entropy},
  \begin{align*}
    \sum_{\chi \colon H(\chi) \ge R} |\widehat{1_S}(\chi)|^3
    &\le \left(\sum_{\chi \colon H(\chi) \ge R} |\widehat{1_S}(\chi)|^2 \right) \sup_{\chi \colon H(\chi) \ge R} |\widehat{1_S}(\chi)| \\
    &\le \left(\sum_{\chi \in \widehat{G}^n} |\widehat{1_S}(\chi)|^2 \right) \sup_{\chi \colon H(\chi) \ge R} \exp(-H(\chi)n/2) \left(\frac{n!}{n^n}\right)^{1/2}\\
    &\le \exp(-R n/2) \left(\frac{n!}{n^n}\right)^{3/2}\\
    &\le \exp((3-R)n/2) \left(\frac{n!}{n^n}\right)^3
  \end{align*}
  as required, where we have used the estimate $n! > (n/e)^n$.
\end{proof}

This is smaller than the main term for any fixed $R > 3$, so from now on such high-entropy characters need not concern us. We refer to such characters as the \emph{high-entropy minor arcs}.  Since a typical character $\chi$ has entropy comparable to $\log n$, the high-entropy case comprises almost all characters in some sense, so this is a good first step.

On the opposite extreme we have characters $\chi$ with entropy $o(1)$, such as $\chi = \left(r^k, -r^k, 0^{n-2k}\right)$ for $k=o(n)$. The characters with entropy $O\(\frac{\log n}{n}\)$ are the major arcs, but between $O\(\frac{\log n}{n}\)$ and $o(1)$ we have the \emph{low-entropy minor arcs}. Such characters are necessarily of the form $(\dots,r^{n-o(n)})$, i.e., they take a single value almost all time. By the remarks above we may shift without loss of generality to assume that this value $r$ is zero. Thus the low-entropy minor arcs are closely related to the set of \emph{sparse} characters, characters $\chi$ comprised almost entirely of zeros.

That leaves the characters $\chi$ with $o(1)\leq H(\chi)\leq 3$, which form the \emph{medium-entropy minor arcs}.  A good model case are characters such as $\chi = (r^{n/3}, - r^{n/3}, 0^{n/3})$, which are particularly troublesome.

There are two fundamental areas of inefficiency in the arguments of Proposition \ref{prop:entropy} and Corollary \ref{cor:crappy-minor-arcs} that prevent them from giving good bounds for low- or medium-entropy minor arcs. First, the bound on $|\hat{1_S}|$ given by Proposition \ref{prop:entropy} is less effective for smaller $H$, and in particular worse than trivial when $H < 1$. Clearly there is no hope for this programme unless we can find a bound on $|\hat{1_S}|$ that is nontrivial throughout the range $0 < H < 1$, and moreover obtains a substantial exponential saving for most of that range.

Second, in the proof of Corollary~\ref{cor:crappy-minor-arcs} we made use of the convenient bound
\[
	\sum_{\chi \in X} |\hat{1_S}(\chi)|^3 \le \left(\sum_{\chi \in \hat{G}} |\widehat{1_S}(\chi)|^2 \right) \sup_{\chi \in X} |\hat{1_S}(\chi)|
\]
for some appropriate set $X\subset\hat{G}^n$. It is fairly clear that we cannot afford this luxury for small entropies: since the main term has order $(n!/n^n)^3$, we need $|\hat{1_S}(\chi)| = o((n!/n^n)^2)$ for all $\chi \in X$ for this to be effective. This is a saving of about $\exp(-n)$ over the trivial bound of $n!/n^n$, and it is simply not reasonable to expect such a bound to hold if, say, $H(\chi) = 1/1000$. One way around this is to pigeonhole the characters $\chi$ into various sets $X_i$, and apply such an $L^2 \cdot L^\infty$ bound on each set. To make this work, we would need a bound on
\[
	\sum_{\chi \in X_i} |\widehat{1_S}(\chi)|^2
\]
which makes a significant saving over the crude estimate
\[
	\sum_{\chi \in X_i} |\widehat{1_S}(\chi)|^2 \le \sum_{\chi \in \widehat{G}^n} |\widehat{1_S}(\chi)|^2
\]
employed in Corollary~\ref{cor:crappy-minor-arcs}. Alternatively, one could try to bound the $L^3$ sum
\[
	\sum_{\chi \in X_i} |\widehat{1_S}(\chi)|^3
\]
directly, using an $L^\infty$ bound and an estimate for the size of $X_i$.

We address both of these inefficiencies in order to reach our final bound. Specifically, we will do each of the following.
\begin{itemize}
  \item Improving on Proposition~\ref{prop:entropy}, we obtain a good general-purpose bound for $|\hat{1_S}(\chi)|$ that is nontrivial for values of $H(\chi)$ approaching zero. Roughly speaking, where Proposition~\ref{prop:entropy} gives the bound $e^{-Hn/2-n/2}$ we will prove the bound $e^{-Hn/2 - n}$. This appears in Section~\ref{sec:srh}.
  \item We use this bound to estimate the contribution to~\eqref{eqn:summing-triples} from $\chi$ with $10^{-10}<H(\chi)<10$, say, by using a dyadic decomposition and a simple estimate for the number of characters of a given entropy. Thus we dispatch the medium-entropy minor arcs. This is covered in Section~\ref{sec:end}.
  \item Separately, we obtain a good estimate for
\[
	\sum_{m\text{-sparse}~\chi} |\widehat{1_S}(\chi)|^2,
\]
i.e., an improvement to Parseval when summing only over those characters $\chi$ having exactly $m$ nonzero terms. Here we might imagine $m \le n/1000$. We call this a \emph{sparseval} bound, and prove this in Section~\ref{sec:sparseval}.
  \item Finally, we obtain a slightly different $L^\infty$ bound specialized to the case of sparse characters $\chi$. This appears in Section~\ref{sec:linfty}. The total contribution from the low-entropy minor arcs is controlled by combining this with the $L^2$ bound above.
\end{itemize}

In summary, we split the universe of all $\chi$ into several slightly overlapping regions -- major arcs, low-entropy minor arcs, medium-entropy minor arcs, and high-entropy minor arcs -- and apply a cocktail of different bounds and explicit computations adapted to each region.

\subsection{Interpretations of minor arc bounds}

As we have said, a significant part of our effort will be spent proving nontrivial bounds on Fourier coefficients $|\hat{1_S}(\chi)|$ for $\chi$ in the minor arcs. Although such bounds are at first glance fairly esoteric, in fact they have natural interpretations. For instance, they are intimately related to questions of the following flavour.

\begin{question}
  Suppose the elements of $G = \Z/n\Z$ are divided into $k$ sets $A_1, \dots, A_k$, each of a pre-determined size $a_i \approx n/k$, at random.  Consider the random variables
  \[
    T_i = \sum_{r \in A_i} r \in G
  \]
  for $i = 1,\dots,k$. The vector $T = (T_1,\dots,T_k)$ is a random variable taking values in the subgroup $H = \{ (x_i) \in G^k \colon \sum x_i = 0 \}$ of $G^k$. How close is $T$ to being equidistributed on $H$, in a quantitative sense? In other words how close is the law of $T$ to the uniform distribution on $H$?
\end{question}

Because of the connection of this question to the size of the Fourier coefficients $\hat{1_S}(r_1^{a_1}, \dots, r_k^{a_k})$, our bounds give a strong answer to this question in many regimes. It is not inconceivable that these kind of bounds may find applications elsewhere.

%
%
%
%

\def\PP{\mathcal{P}}

\section{Major arcs}\label{sec:major-arcs}

In this section we compute $\hat{1_S}(r_1,\dots,r_m,0,\dots,0)$ explicitly whenever $m$ is bounded. To this end observe that
\begin{equation} \label{eqn:sparse-identity}
	\hat{1_S}(r_1,\dots,r_m,0,\dots,0) = \frac{(n-m)!}{n^n} \sump_{x_1,\dots,x_m} e(-r_1x_1-\cdots-r_mx_m).
\end{equation}
The sum over distinct $x_1,\dots,x_m$ can be related to a sum over partitions of $\{1,\dots,m\}$ using a type of M{\"o}bius inversion: for any function $F(x_1,\dots,x_m)$ we have
\[
	\sump_{x_1,\dots,x_m} F(x_1,\dots,x_m) = \sum_{\PP} \mu(\PP) \sum_{\substack{x_1,\dots,x_m\\ x_i = x_j~\text{whenever}~i\stackrel{\PP}{\sim} j}} F(x_1,\dots,x_m),
\]
where the outer sum runs over partitions $\PP$ of $\{1,\dots,m\}$, and
\[
	\mu(\PP) = (-1)^{m-|\PP|} \prod_{P\in\PP} (|P|-1)!.
\]
See, e.g.,~\cite{handbookenumerative} for more information. To apply this to $\hat{1_S}$, say that a partition $\PP$ of $\{1,\dots,m\}$ \emph{kills} $(r_1,\dots,r_m)$ if $\sum_{i\in P} r_i = 0$ for each $P\in\PP$, and observe that
\[
	\sum_{\substack{x_1,\dots,x_m\\ x_i = x_j~\text{whenever}~i\stackrel{\PP}{\sim} j}} e(-r_1x_1-\cdots-r_mx_m) = 
	\begin{cases}
		n^{|\PP|} & \text{if}~\PP~\text{kills}~(r_1,\dots,r_m),\\
		0 & \text{otherwise},
	\end{cases}
\]
so
\begin{equation}\label{partitionformula}
	\hat{1_S}(r_1,\dots,r_m,0,\dots,0) = \frac{(n-m)!}{n^n} \sum_{\PP~\text{killing}~(r_1,\dots,r_m)} \mu(\PP) n^{|\PP|}.
\end{equation}

The following two observations are immediate from this formula:
\begin{itemize}
	\item Suppose every killing partition of $(r_1,\dots,r_m)$ has at most $k$ parts. Then
\begin{equation}
	|\hat{1_S}(r_1,\dots,r_m,0,\dots,0)| \leq O_m\( \frac1{n^{m-k}} \frac{n!}{n^n}\).\label{nmkbound}
\end{equation}
	\item Suppose that $m$ is even, that $r_1,\dots, r_m\in\hat{G}$ are nonzero, and that $(r_1,\dots,r_m)$ is killed by a unique partition with $m/2$ parts. In other words suppose that $r_1,\dots,r_m\in\hat{G}$ are distinct and nonzero, and that, up to a permutation, $r_{2j}=-r_{2j-1}$ for each $j=1,\dots,m/2$. Then
\begin{equation}
	\hat{1_S}(r_1,\dots,r_m,0,\dots,0) = \frac{(-1)^{m/2}}{n^{m/2}}\frac{n!}{n^n} + O_m\(\frac1{n^{m/2+1}} \frac{n!}{n^n}\).\label{pairsterms}
\end{equation}
\end{itemize}

\begin{proposition}\label{majorarcs}
If $m$ is even then
\[
	\sum_{m\textup{-sparse}~\chi} \hat{1_S}(\chi)^3 = \frac{(-1)^{m/2}}{2^{m/2} (m/2)!} \frac{n!^3}{n^{3n}} + O_m\(\frac1n \frac{n!^3}{n^{3n}}\),
\]
while if $m$ is odd then
\[
	\sum_{m\textup{-sparse}~\chi} \hat{1_S}(\chi)^3 = O_m\(\frac1n \frac{n!^3}{n^{3n}}\).
\]
\end{proposition}
\begin{proof}
First of all note by permutation-invariance that
\begin{equation}\label{nchoosemfactor}
	\sum_{m\text{-sparse}~\chi} \hat{1_S}(\chi)^3 = \binom{n}{m} \sum_{r_1,\dots,r_m\neq 0} \hat{1_S}(r_1,\dots,r_m,0,\dots,0)^3.
\end{equation}
For each $(r_1,\dots,r_m)$ choose a maximal-size partition $\PP$ which kills $(r_1,\dots,r_m)$, and split the sum up according to $\PP$. Suppose $\PP$ has $k$ parts. Since $\PP$ cannot have singletons we have $k\leq m/2$, and by~\eqref{nmkbound} we have
\[
	|\hat{1_S}(r_1,\dots,r_m,0,\dots,0)| \leq O_m\( \frac1{n^{m-k}} \frac{n!}{n^n}\).
\]
Since the number of $(r_1,\dots,r_m)$ killed by $\PP$ is bounded by $n^{m-k}$, the contribution to~\eqref{nchoosemfactor} from these $(r_1,\dots,r_m)$ is bounded by
\[
	\binom{n}{m} n^{m-k} \(\frac{O_m(1)}{n^{m-k}} \frac{n!}{n^n}\)^3 = O_m\(\frac1{n^{m-2k}} \frac{n!^3}{n^{3n}}\),
\]
which is satisfactory unless $k=m/2$. Moreover the number of $(r_1,\dots,r_m)$ killed by at least two partitions $\PP$ with $m/2$ parts is bounded by $n^{m/2-1}$, so the contribution from these $(r_1,\dots,r_m)$ is bounded by
\[
	\binom{n}{m} n^{m/2-1} O_m\(\frac1{n^{m/2}} \frac{n!}{n^n}\)^3 = O_m\(\frac1{n} \frac{n!^3}{n^{3n}}\),
\]
which is again satisfactory. Thus we may restrict our attention to those $(r_1,\dots,r_m)$ which are killed by a unique partition $\PP$ with $m/2$ parts, and for these~\eqref{pairsterms} gives
\[
	\hat{1_S}(r_1,\dots,r_m,0,\dots,0) = \frac{(-1)^{m/2}}{n^{m/2}}\frac{n!}{n^n} + O_m\(\frac1{n^{m/2+1}} \frac{n!}{n^n}\).
\]
For a fixed such $\PP$ the number of such $(r_1,\dots,r_m)$ is $n^{m/2} + O_m(n^{m/2-1})$, and the number of such $\PP$ is
\[
	(m-1)(m-3)\cdots 1 = \frac{m!}{2^{m/2}(m/2)!},
\]
so the total contribution from these $(r_1,\dots,r_m)$ is
\begin{align*}
	\binom{n}{m} \frac{m!}{2^{m/2}(m/2)!} (n^{m/2} + O_m(n^{m/2-1})) & \(\frac{(-1)^{m/2}}{n^{m/2}}\frac{n!}{n^n} + O_m\(\frac1{n^{m/2+1}} \frac{n!}{n^n}\)\)^3\\
	&= \frac{(-1)^{m/2}}{2^{m/2}(m/2)!} \frac{n!^3}{n^{3n}} + O_m\(\frac1{n} \frac{n!^3}{n^{3n}}\).
\end{align*}
This proves the proposition.
\end{proof}

\section{Square-root cancellation for general Fourier coefficients}\label{sec:srh}

The aim of this section is to prove the following refinement of Proposition~\ref{prop:entropy}.

\begin{theorem}\label{SRH}
Suppose $\chi=(r_1^{a_1},\dots,r_k^{a_k})$, where $\sum_{i=1}^k a_i = n$. Then
\[
	|\hat{1_S}(\chi)| \leq \binom{n+k-1}{k-1}^{1/2} \binom{n}{a_1,\dots,a_k}^{-1/2} \frac{n!}{n^n}.
\]
\end{theorem}

\begin{remark} \ 
  \begin{itemize}
    \item The terms depending on $k$ are easily seen to be a small error in the regimes we care about.  Indeed, if $H(\chi) = O(1)$ then $\chi$ can take at most $O(n/\log n) = o(n)$ distinct values, and the additional term is then $\exp(o(n))$, and thus small compared to the exponential entropy saving.

    \item The saving over Proposition \ref{prop:entropy} is an apparently modest factor of around $e^{-n/2 + o(n)}$. This may be uninspiring if $H$ is large, but if $H$ is small it is decisive.

    \item In general this bound is not sharp, often by a substantial amount. However, it is sometimes attained asymptotically for very special, arithmetically structured classes of characters. For instance, if we temporarily allow $n$ to be even, then one can compute directly that
      \[
        \log \left[ |\widehat{1_S}((1/2)^a, 0^{n-a})| / (n!/n^n) \right] \sim -\frac1{2} \log \binom{n}{a}
      \]
      for large $n$.  Similarly, numerical evidence strongly suggests that
      \[
        \log \left[ |\widehat{1_S}((1/3)^a, (-1/3)^a, 0^{n-2a})| / (n!/n^n) \right] \sim -\frac1{2} \log \binom{n}{a,a,n-2a}
      \]
      if $n$ is divisible by $3$; for example, when $n=3003$,
      \[
      \log \left[ |\hat{1_S}((1/3)^{1001},(-1/3)^{1001},0^{1001})| / (n!/n^n) \right] = -1649.01782245...,
      \]
      while
      \[
      - \frac12 \log \binom{3003}{1001,1001,1001} = -1645.46757758....
      \]
      But, for example, the evidence also suggests that
      \[
        \log \left[ |\widehat{1_S}((1/3)^a, 0^{n-a})| /(n!/n^n) \right] \sim -\frac{2}{3} \log \binom{n}{a}.
      \]
  \end{itemize}
\end{remark}

From the above remarks, we see that one might hope to improve on Theorem \ref{SRH}, but only by ruling out the particular characters discussed above and others similar to them.  For instance, a strengthening is almost certainly true under the assumption that $n$ is prime. Thus one might think of Theorem~\ref{SRH} as the best general-purpose bound available.

We should also comment briefly on the term ``square-root cancellation''.  If $\chi = (r_1^{a_1}, \dots, r_k^{a_k})$ then
\[
  \widehat{1_S}(\chi) = \frac{n!}{n^n} \mathop{{} \bigE}_{\text{partitions } \mathcal{P}} \ e\left( -r_1 \left(\sum_{x \in P_1} x\right) - \dots - r_k\left(\sum_{x \in P_k} x\right) \right)
\]
where the expectation is over all ordered partitions $\mathcal{P} = (P_1, \dots, P_k)$ of $G$ into $k$ pieces with $|P_i| = a_i$. The number of such partitions is precisely $\binom{n}{a_1,\dots,a_k}$; assuming that the phases in the average behave randomly then we should expect square-root cancellation, and we heuristically recover the bound in Theorem~\ref{SRH} up to lower-order terms.

\begin{proof}[Proof of Theorem~\ref{SRH}]
Let $\gamma$ be the Gaussian measure on $\C^k$ defined with respect to Lebesgue measure $\lambda$ by
\[
  \frac{d\gamma}{d\lambda} = \frac1{\pi^k} \exp\(-\sum_{i=1}^k |z_i|^2\).
\]
We claim that
\begin{equation} \label{eqn:magic-srh}
	n^n \hat{1_S}(\chi)
	= \int_{\C^k} z_1^{a_1}\cdots z_k^{a_k} \prod_{x\in G}\( \sum_{i=1}^k e(-r_i x) \overline{z_i}\)\, d\gamma.
\end{equation}
One can check this by expanding the product and using the identity
\begin{equation}\label{aibirelation}
	\int_{\C^k} \prod_{i=1}^k z_i^{a_i} \overline{z_i}^{b_i}\,d\gamma = \begin{cases} \prod_{i=1}^k a_i! & \text{if}~a_i=b_i~\text{for each}~i,\\ 0 & \text{otherwise}.\end{cases}
\end{equation}

We can now proceed by bounding the integral in~\eqref{eqn:magic-srh} using various techniques.  Applying the Cauchy--Schwarz inequality we bound the right hand side by
\begin{equation}\label{CSbound}
	\(\int_{\C^k} |z_1|^{2a_1} \cdots |z_k|^{2a_k} \,d\gamma\)^{1/2} \(\int_{\C^k} \prod_{x\in G}\left|\sum_{i=1}^k e(-r_i x) \overline{z_i}\right|^2\,d\gamma\)^{1/2}.
\end{equation}
The first factor here is exactly $\(\prod_{i=1}^k a_i!\)^{1/2}$ by~\eqref{aibirelation}. If we now apply the AM--GM inequality to the second term and evaluate the resulting integral, we get
\begin{align*}
	\int_{\C^k} \prod_{x\in G}\left| \sum_{i=1}^k e(-r_i x)\overline{z_i}\right|^2\,d\gamma
	&\leq \int_{\C^k} \(\frac1n \sum_{x\in G} \left|\sum_{i=1}^k e(-r_i x) \overline{z_i}\right|^2\)^n\,d\gamma\\
	&= \int_{\C^k} \(\sum_{i=1}^k |z_i|^2\)^n\,d\gamma\\
	&= \sum_{s_1+\cdots+s_k = n} \binom{n}{s_1,\dots,s_k} \int_{\C^k} |z_1|^{2s_1}\cdots|z_k|^{2s_k}\,d\gamma\\
	&= n! \sum_{s_1+\cdots+s_k = n} 1\\
	&= n! \binom{n+k-1}{k-1}.
\end{align*}
Thus the theorem follows by dividing~\eqref{CSbound} by $n^n$.
\end{proof}

\begin{remark}
  Though the identity~\eqref{eqn:magic-srh} is easily verified directly, we briefly sketch here a possible route by which one might be motivated to write down such a formula in the first place.

  Let $\{X_x \colon x \in G\}$ be indeterminates, and observe that the left hand side of~\eqref{eqn:magic-srh} is precisely the coefficient of $\prod_{x \in G} X_x$ in the multivariate polynomial
  \[
    \prod_{i=1}^k \left( \sum_{x \in G} X_x e(-r_i x) \right)^{a_i},
  \]
  and so equivalently the value of the derivative
  \begin{equation} \label{eqn:srh-deriv}
    \left(\prod_{x \in G} \frac{\partial}{\partial X_x}\right) \prod_{i=1}^k \left( \sum_{x \in G} X_x e(-r_i x) \right)^{a_i}
  \end{equation}
  evaluated at zero (although in fact this evaluation is redundant as this expression is a constant function).

  In some generality, it is possible to equate a quantity of the form
  \[
    \frac{\partial^{b_1}}{\partial Y_1^{b_1}} \dots \frac{\partial^{b_k}}{\partial Y_k^{b_k}} f \big|_{Y_1 = \dots = Y_k = 0},
  \]
  where $f(Y_1, \dots, Y_k)$ is a suitable holomorphic function of $k$ complex variables, with the corresponding integral expression
  \[
    \int_{\C^k} \overline{Y_1^{b_1}} \dots \overline{Y_k^{b_k}}\, f(Y_1, \dots, Y_k) \, d \gamma,
  \]
  where again $\gamma$ denotes Gaussian measure. This can be verified by applying the case $k=1$ iteratively, which in turn follows by averaging the usual Cauchy integral formula over different radii. This formula can be considered a particular higher-dimensional variant of the Cauchy integral formula.
  
  The identity~\eqref{eqn:magic-srh} follows by applying this identity to~\eqref{eqn:srh-deriv} and making an orthogonal change of variables.

%
\end{remark}


\section{Sparseval}\label{sec:sparseval}

We recall from Section~\ref{idiotsguide} that one of our goals is to obtain good bounds on
\begin{equation} \label{eqn:hard-sparse}
  \sum_{m\text{-sparse}~\chi} |\widehat{1_S}(\chi)|^2,
\end{equation}
where the sum is over all characters $\chi=(r_1,\dots,r_n)$ with precisely $m$ nonzero entries.

First consider the related sum
\begin{equation} \label{eqn:easy-sparse}
  \sum_{\leq m\text{-sparse}~\chi} |\widehat{1_S}(\chi)|^2,
\end{equation}
i.e., the sum over characters with \emph{at most} $m$ non-zero entries.  This can be bounded by applying Parseval to the set $S_m$ of injections $\{1,\dots,m\}\to G$ as a subset of the group $G^m$. Indeed, letting $N(\chi)$ be the set of indices $i$ for which $r_i$ is nonzero, we have
\begin{align}
\sum_{\leq m\text{-sparse}~\chi} |\hat{1_S}(\chi)|^2
&\leq \sum_{|N|=m} \sum_{N(\chi) \subset N} |\hat{1_S}(\chi)|^2\nonumber\\
&= \binom{n}{m} \sum_{r_1,\dots,r_m} |\hat{1_S}(r_1,\dots,r_m,0,\dots,0)|^2\nonumber\\
&= \binom{n}{m} \sum_{r_1,\dots,r_m} \frac{(n-m)!^2}{n^{2n-2m}} |\hat{1_{S_m}}(r_1,\dots,r_m)|^2\label{easysparseval}\\
&= \binom{n}{m} \frac{(n-m)!^2}{n^{2n-2m}} \frac{n!}{n^m(n-m)!}\nonumber\\
&= \frac{n^m}{m!} \frac{n!^2}{n^{2n}}.\nonumber
\end{align}
Here we used (\ref{eqn:sparse-identity}).

This is our first nontrivial \emph{sparseval} bound. It turns out that overcounting as above is too inefficient in the context of the wider argument, when, say, $m/n$ is a small constant. The purpose of the rest of this section therefore is to improve the bound on~\eqref{eqn:hard-sparse} as much as possible.

\begin{theorem}\label{sparseval} If $m\leq n/2$ then
\[
	\sum_{m\textup{-sparse}~\chi} |\hat{1_S}(\chi)|^2 \leq O(m^{1/4}) e^{O(m^{3/2}/n^{1/2})} \binom{n}{m}^{1/2} \frac{n!^2}{n^{2n}}
\]
\end{theorem}
\begin{proof}
The starting point of our strategy is to apply inclusion--exclusion to obtain an expression for~\eqref{eqn:hard-sparse}.

For $m\leq n$ let
\[
	Q(m,n) = \frac{n^{2n}}{n!^2} \sum_{m\text{-sparse}~\chi} |\hat{1_S}(r)|^2,
\]
and observe as in \eqref{easysparseval} that
\begin{align*}
	\frac{n^m}{m!} 
	&= \frac{n^{2n}}{n!^2} \sum_{|N|=m} \sum_{N(\chi)\subset N} |\hat{1_S}(\chi)|^2\\
	&= \frac{n^{2n}}{n!^2} \sum_{k=0}^m \sum_{|N(\chi)| = k} |\hat{1_S}(\chi)|^2 \binom{n-k}{m-k}\\
	&= \sum_{k=0}^m \binom{n-k}{m-k} Q(k,n).
\end{align*}
By inverting this relation we have
\begin{equation}\label{Qmnsum}
	Q(m,n) = \sum_{k=0}^m (-1)^{m-k} \binom{n-k}{m-k} \frac{n^k}{k!} = \frac1{m!} \sum_{k=0}^m \binom{m}{k} (-1)^k (n-m+1)\cdots(n-k) n^k.
\end{equation}

This expression exhibits a vast amount of cancellation. To exploit this, we employ a generating function argument, followed by some complex analysis in the spirit of the saddle-point method (see~\cite[Chapter~VIII]{flajolet-sedgewick} for background); see also Remark \ref{remarkonmagicf} for some further intuition. 

Observe from~\eqref{Qmnsum} that $Q(m,n)$ is exactly the coefficient of $X^m$ in
\begin{equation}\label{magicf}
	f(X) = \frac{n^{n+1} e^X}{(n+X)^{n-m+1}}.
\end{equation}
Thus by Cauchy's formula
\[
	Q(m,n) = \frac{n^{n+1}}{i2\pi} \oint_{|z|=r} \frac{e^z}{(n+z)^{n-m+1} z^{m+1}} \, dz,
\]
for any $r$ in the range $0<r<n$, so
\[
	Q(m,n) \leq \frac{n^{n+1}}{r^m} \max_{|z|=r} \left|\frac{e^z}{(n+z)^{n-m+1}}\right|.
\]
We claim that $\left|e^z/(n+z)^{n-m+1}\right|$ has only two local maxima on $|z|=r$: one at $z=+r$ and the other at $z=-r$. To see this note that if $|z|=r$ and $\Re z = t$ then
\[
	\left|\frac{e^z}{(n+z)^{n-m+1}}\right|^2 = \frac{e^{2t}}{(n^2+r^2 + 2nt)^{n-m+1}},
\]
so
\[
	\frac{d}{dt}\log\left|\frac{e^z}{(n+z)^{n-m+1}}\right|^2 = 2 - \frac{2n(n-m+1)}{n^2+r^2+2nt}.
\]
This function has a unique pole at $t= -(n+r^2/n)/2$, which is to the left of the allowed region $-r\leq t \leq r$, and it has a unique zero somewhere corresponding to a minimum, not a maximum, so the claim holds. Thus $\left|e^z/(n+z)^{n-m+1}\right|$ is bounded on $|z|=r$ by
\[
	\frac{e^{\pm r}}{(n\pm r)^{n-m+1}},
\]
so
\[
	Q(m,n)
	\leq \frac{n^{n+1} e^{\pm r}}{r^m (n\pm r)^{n-m+1}}
	= \frac{e^{\pm r}}{(1\pm r/n)^{n-m+1}} \frac{n^m}{r^m}.
\]
Here
\begin{align*}
	\log\left(\frac{e^{\pm r}}{(1\pm r/n)^{n-m+1}}\right)
	&= \pm r - (n-m+1)(\pm r/n - r^2/2n^2 + O(r^3/n^3))\\
	&= r^2/2n + O(r^3/n^2 + rm/n).
\end{align*}
Taking $r = (mn)^{1/2}$ we thus have, by Stirling's formula,
\[
	Q(m,n) \leq e^{(1/2 + O((m/n)^{1/2}))m} \frac{n^{m/2}}{m^{m/2}} \leq O(m^{1/4}) e^{O(m^{3/2}/n^{1/2})} \binom{n}{m}^{1/2}.\qedhere
\]
\end{proof}

\begin{remark}\label{remarkonmagicf}
Although it is straightforward to verify from~\eqref{Qmnsum} that $Q(m,n)$ is the coefficient of $X^m$ in~\eqref{magicf}, this proof is unsatisfying in that one needs to know the formula for $f$ in advance. Here is an alternative proof of this fact which does not have this fault.

Let $\partial$ be the forward difference operator defined on functions $g:\N\to\R$ by
\[
	\partial g (k) = g(k+1)-g(k),
\]
and observe that
\[
	Q(m,n) = \frac1{m!} \partial^m g_{m,n}(0),
\]
where $g_{m,n}$ is the polynomial defined by
\[
	g_{m,n}(k) = n^k (n-k)(n-k-1)\cdots(n-m+1).
\]
Intuitively, the sum~\eqref{Qmnsum} exhibits enormous cancellation, precisely because it is an expression for a high derivative of a very ``smooth'' function $g_{m,n}$. To take advantage of this smoothness we investigate the first derivative $\partial g_{m,n}$. We compute
\begin{align*}
	\partial g_{m,n}(k)
	&= n^{k+1} (n-k-1) (n-k-2)\cdots(n-m+1) - n^k (n-k)(n-k-1)\cdots(n-m+1)\\
	&= k n^k (n-k-1) (n-k-2)\cdots(n-m+1)\\
	&= \frac{k}{n} g_{m,n}(k+1).
\end{align*}
From this it is possible to control the higher derivatives: using the Leibniz-type formula
\[
	\partial(gh)(k) = \partial g(k) h(k+1) + g(k) \partial h(k)
\]
we derive the recurrence
\begin{align*}
	\partial^\ell g_{m,n}(k)
	&= \partial^{\ell-1} \(\frac{k}{n} g_{m,n}(k+1)\)\\
	&= \frac{k}{n} \partial^{\ell-1} g_{m,n}(k+1) + \frac{\ell-1}{n} \partial^{\ell-2} g_{m,n}(k+2),
\end{align*}
which holds for $\ell\geq 2$. In particular, letting
\[
	a_\ell = \frac1{\ell!} \partial^\ell g_{m,n}(m-\ell),
\]
we have the recurrence
\begin{equation}\label{aell-recurrence}
	\ell a_\ell = \frac{m-\ell}n a_{\ell-1} + \frac1n a_{\ell-2}
\end{equation}
for $(a_\ell)_{\ell\geq 0}$. Observe that $a_0 = n^m$, $a_1=(m-1)n^{m-1}$, and that $Q(m,n) = a_m$.

In Section~\ref{sec:linfty} we will bound solutions to recurrences like~\eqref{aell-recurrence} in a direct fashion. For~\eqref{aell-recurrence} itself, it is convenient to employ a generating function technique. For an indeterminate $X$ put
\[
	f(X) = \sum_{\ell=0}^\infty a_\ell X^\ell,
\]
and observe by multiplying~\eqref{aell-recurrence} by $X^{\ell-1}$ and summing for $\ell\geq 2$ that
\[
	f'(X) - a_1 = \frac{m-1}{n} (f(X)-a_0) - \frac1n Xf'(X) + \frac1n X f(X) \ .
\]
By inputting $a_0 = n^m$ and $a_1 = (m-1)n^{m-1}$ and rearranging we arrive at
\[
	(n-X) f'(X) = (m-1+X) f(X).
\]
The formula~\eqref{magicf} is obtained by solving this differential equation.
\end{remark}

\section{An \texorpdfstring{$L^\infty$}{L-infty} bound for low-entropy minor arcs}
\label{sec:linfty}

The main result of this section is the following proposition.

\begin{proposition}\label{rudiLinfty}
Let $\chi$ be a character with exactly $m$ nonzero coordinates, where $m \leq n/3$. Then
\[
  |\hat{1_S}(\chi)| \leq e^{O(m^{3/2}/n^{1/2} + m^{1/2})}\cdot 2^{-m/2}\binom{n}{m}^{-1/2} \cdot \frac{n!}{n^n}.
\]
\end{proposition}

\begin{remark}
  It is instructive to compare this bound to what one can deduce from Theorem~\ref{SRH}.  If $\chi$ has exactly $m$ nonzero coordinates and takes precisely $k$ distinct nonzero values, i.e., $\chi = (r_1^{a_1}, \dots, r_k^{a_k}, 0^{n-m})$ where $\sum a_i = m$, then Theorem~\ref{SRH} gives us a bound of
  \[
    |\hat{1_S}(\chi)| \leq \binom{n+k}{k}^{1/2} \binom{n}{a_1,\dots,a_k,n-m}^{-1/2} \frac{n!}{n^n}
  \]
  and the worst case of this bound over all choices of $k \le m$ and $a_1, \dots, a_k$ can be computed approximately as
  \[
    |\hat{1_S}(\chi)| \leq \binom{n}{m}^{-1/2} e^{o(m)} \frac{n!}{n^n},
  \]
  provided that $m$ is significantly greater than $\sqrt{n}$. In other words a character similar to $\chi = (r^{m}, 0^{n-m})$ is asymptotically about as bad as the worst case. Moreover if we allow $n$ to be even and set $r=1/2$ then the bound from Theorem~\ref{SRH} is essentially sharp, as remarked in that section.
  
  However, in this setting, Proposition~\ref{rudiLinfty} improves on this bound by a significant factor of $2^{-m/2}$. Clearly then such an improvement is only possible under the assumption that $n$ is odd, and the proof of Proposition~\ref{rudiLinfty} must exploit this assumption in a fundamental way.
  Since we do not know how to adapt the proof of Theorem~\ref{SRH} to exploit the absence of $2$-torsion, we employ a more specialized approach that suffices in the sparse regime.
  
We observe that this is the unique stage of the proof in which we need the full strength of the hypothesis that $G$ has odd order. Elsewhere we only need to assume that $\sum_{x\in G} x = 0$.
\end{remark}

\begin{proof}[Proof of Proposition~\ref{rudiLinfty}]

  Let $\chi = (r_1, \dots, r_m, 0^{n-m})$.  The key tool in our proof is an exact recursive formula for Fourier coefficients $\widehat{1_S}(\chi)$, in terms of related Fourier coefficients for which $m' < m$.  Specifically, as long as $r_m$ is nonzero we have
  \begin{align}  \label{fivepointone}
    \hat{1_S}(\chi)
    &= \frac{(n-m)!}{n^n}\sump_{x_1,\dots,x_m} e(-r_1x_1-\cdots-r_mx_m)\nonumber\\
    &=\frac{(n-m)!}{n^n}\sump_{x_1,\dots,x_{m-1}}\sum_{x_m\neq x_1,\dots,x_{m-1}} e(-r_1x_1-\cdots-r_mx_m)\nonumber\\
    &=-\frac{(n-m)!}{n^n}\sump_{x_1,\dots,x_{m-1}} \sum_{x_m=x_1,\dots,x_{n-1}} e(-r_1x_1-\cdots-r_mx_m)\nonumber\\
    &=-\frac1{n-m+1}\sum_{i=1}^{m-1} \hat{1_S}(r_1,\dots,r_i+r_m,\dots,r_{m-1},0,\dots,0).
  \end{align}
  As an aside, we remark that this formula can be used for efficient explicit computation of certain kinds of Fourier coefficient.

  Let $U_m$ be the maximal value of the $|\hat{1_S}|$ taken over all characters with exactly $m$ nonzero coordinates.  If we apply the triangle inequality to the right hand side of (\ref{fivepointone}) and bound each $|\widehat{1_S}(\chi')|$ appearing there by the appropriate value $U_{m'}$ (where $m' < m$), we obtain recursive bounds on the $U_m$.

  However, there is a subtlety in this process: the number of nonzero coordinates of
  \[
    (r_1,\dots,r_i+r_m,\dots,r_{m-1},0,\dots,0)
  \]
  might be either $(m-1)$ (if $r_i + r_m \ne 0$) or $(m-2)$ (if $r_i + r_m = 0$), and we do not have much control over which occurs.  Since we expect $U_{m-1} < U_{m-2}$, it is in our interests to land in the former case as much as possible.  We have the freedom to reorder the $r_i$ before applying (\ref{fivepointone}), so our goal is to optimize the recursive bound over all choices of $r_m$.

  For each $i\leq m$ let $\mathcal{N}_i=\{j\leq m\colon r_j = -r_i\}$. By reordering, we may assume without loss of generality that $|\mathcal{N}_i|$ is minimized when $i=m$. Suppose $j\in \mathcal{N}_m$. Then $r_j = -r_m$, so since $|G|$ is odd we have $r_j\neq r_m$, so $\mathcal{N}_m\cap \mathcal{N}_j=\emptyset$. Thus by minimality $|\mathcal{N}_m|\leq m/2$. 

  Hence, applying \eqref{fivepointone} we deduce the bound
\begin{align*}
	|\widehat{1_S}(\chi)|
  &\leq \frac{|\mathcal{N}_m| U_{m-2} + (m-1-|\mathcal{N}_m|) U_{m-1}}{n-m+1}\\
  &\leq \frac1{n-m+1} \max\left\{ (m-1) U_{m-1},\  (m/2) U_{m-2} + (m/2-1) U_{m-1}\right\}
\end{align*}
where the former term in the maximum covers the case $U_{m-1} > U_{m-2}$ and the latter the more probable case $U_{m-1} \le U_{m-2}$. Hence
\begin{equation}  \label{um-recursion}
	U_m \le
  \frac1{n-m+1} \max\left\{ (m-1) U_{m-1},\  (m/2)U_{m-2} + (m/2-1) U_{m-1}\right\}.
\end{equation}

Our task is now to obtain bounds on $U_m$ by solving this recurrence, which -- ignoring the maximum -- resembles a kind of time-dependent Fibonacci sequence.

In Remark~\ref{remarkonmagicf} we dealt with a similar recurrence using generating function methods.  Here, the presence of the maximum of two alternatives, and possibly other reasons, make this approach less attractive.  Instead, we will bound \eqref{um-recursion} by more hands-on methods.

Specifically, we will phrase \eqref{um-recursion} in terms of a product of $2 \times 2$ matrices (as one might for the Fibonacci sequence), and control the $L^2$ norm of the result by bounding the $L^2 \to L^2$ operator norms of each matrix in the product.

Write $\alpha_m = \textstyle\frac{m}{2(n-m+1)}$, $\beta_m = \textstyle\frac{m-2}{2(n-m+1)}$, and $\gamma_m = \textstyle\frac{m-1}{n-m+1}$. Additionally, we work with a rescaled version $V_m=(\alpha_1\alpha_2\dots \alpha_m)^{-1/2} U_m$ of $U_m$, for $m\geq 1$.  From (\ref{um-recursion}) we deduce that
\begin{equation}\label{E:Vmrecursion}
  V_m \leq \max \left\{ \gamma_m\alpha_m^{-1/2}\, V_{m-1},\  \alpha_m^{1/2}\alpha_{m-1}^{-1/2}\, V_{m-2} + \beta_m\alpha_{m}^{-1/2}\, V_{m-1} \right \}
\end{equation}
holds for any $m \ge 3$.

If we define matrices
\[
 M_m = \left( \begin{array}{cc}
 \gamma_m\alpha_m^{-1/2} & 0 \\
 1 & 0
 \end{array} \right) \quad\text{and}\quad N_m = \left( \begin{array}{cc}
 \beta_m\alpha_m^{-1/2} & \alpha_m^{1/2}\alpha_{m-1}^{-1/2} \\
 1 & 0
 \end{array} \right),
\]
and vectors $v_m = (V_m, V_{m-1})^T$, we can write \eqref{E:Vmrecursion} equivalently as
\[
  v_m \leq \max\left\{ M_mv_{m-1}, N_mv_{m-1} \right\}.
\]

Using the easily obtainable bounds $\alpha_m\alpha_{m-1}^{-1} = 1+O(1/m)$ and $\beta_m^2\alpha_m^{-1}\leq m/n$, we get that
\[
  \trace\left(N_m^T N_m\right) \leq 2 + m/n + O(1/m)
\]
and
\[
  \det\left(N_m^T N_m\right) = 1 + O(1/m),
\]
and so by considering the singular values of $N_m$, we get the operator norm bound
\[
  \|N_m\|_{L^2 \to L^2} \leq 1+O((m/n)^{1/2} + m^{-1/2}).
\]

Completely analogously, using the easy bound $\gamma_m^2\alpha_m^{-1}\leq 4m/n$, we deduce
\[
  \|M_m\|_{L^2 \to L^2} \leq 1+O(m/n).
\]

Thus we have
\begin{align*}
  V_m 
  &\leq \sqrt{V_m^2 + V_{m-1}^2} \\
  &\leq \left( \prod_{r=3}^m \max\left\{ \|M_r\|_{L^2 \to L^2}, \|N_r\|_{L^2 \to L^2} \right\} \right) \cdot \sqrt{V_2^2+V_1^2} \\
  &\leq \left( \prod_{r=3}^m \left(1+O\left((r/n)^{1/2} + r^{-1/2} \right) \right) \right) \cdot \sqrt{V_2^2+V_1^2}.
\end{align*}

We already know from Section~\ref{idiotsguide} that $U_1=0$ and it is evident from \eqref{fivepointone} that $U_2 =O(n!/n^{n+1})$. Hence, we can bound the term $\sqrt{V_2^2+V_1^2}$ in the above inequality by $O(n!/n^n)$, which gives
\[
  V_m \leq e^{O\left(m^{3/2}/n^{1/2} + m^{1/2}\right)}\cdot \frac{n!}{n^{n}},
\]
and the claim stated in the proposition follows easily from this.
\end{proof}

\section{End of the argument}\label{sec:end}

To estimate the sum of cubes
\[
	\sum_{\chi\in\hat{G}^n} \hat{1_S}(\chi)^3
\]
we divide the set of all $\chi\in\hat{G}^n$ into three regions depending on the value of $H(\chi)$: a high-entropy range $H\geq 10$, a medium-entropy range $\eps\leq H \leq 10$, and a low-entropy range $H\leq \eps$. Here $\eps$ is a small positive constant which will be chosen at the end of the argument.

In the high-entropy range $H\geq 10$ it is enough to use the bound from Corollary \ref{cor:crappy-minor-arcs},
\[
	\sum_{\chi \colon H(\chi)\geq 10} |\hat{1_S}(\chi)|^3
	\leq e^{-3.5n} n!^3/n^{3n}.
\]

In the medium-entropy range $\eps\leq H\leq 10$, we first need a bound for the number of characters of a given entropy.
	
\begin{lemma}
  \label{number-of-characters-H}
If $H\leq 10$ then the number of characters of entropy at most $H$ is bounded by $e^{Hn + o(n)}$.
\end{lemma}
\begin{proof}
Every character of entropy at most $H$ has an orbit under the permutation action of size at most $e^{Hn}$, so it suffices to show that there are only $e^{o(n)}$ such orbits of characters of entropy at most $10$. Every orbit is uniquely specified by giving the number $a_r$ of appearances of $r$ for each $r\in\hat{G}$, so we must show that the number of nonnegative integer vectors $(a_r)_{r\in\hat{G}}$ with $\sum_r a_r = n$ and satisfying
\[
	\frac{n!}{\prod_{r\in\hat{G}} a_r!} \leq e^{10n}
\]
is $e^{o(n)}$. Let $\delta>0$ be a small constant, and let $t$ be the sum of the $a_r$ for which $a_r \leq e^{-11/\delta}n$. Then
\[
	e^{-11n} n^n \leq e^{-10n} n! \leq \prod_{r\in\hat{G}} a_r! \leq \prod_{r\in\hat{G}} a_r^{a_r} \leq (e^{-11/\delta} n)^t n^{n-t} = e^{-11t/\delta} n^n,
\]
so $t\leq \delta n$. Since at most $e^{11/\delta}$ of the $a_r$ are bigger than $e^{-11/\delta} n$, the number of $(a_r)_{r\in\hat{G}}$ is bounded by the number of ways of choosing the set $B$ of at most $e^{11/\delta}$ indices $r$ for which $a_r$ is big, times the number of ways of choosing these $a_r$, times the number of ways of choosing $(a_r)_{r\notin B}$ so that $\sum_{r\notin B} a_r = n - \sum_{r\in B} a_r \leq \delta n$. This is bounded by
\[
	n^{e^{11/\delta}} n^{e^{11/\delta}} \binom{n+\delta n-1}{\delta n - 1}.
\]
The claimed estimate now follows by applying Stirling's formula and choosing $\delta$ appropriately.
\end{proof}
	
We conclude by combining Lemma~\ref{number-of-characters-H} with Theorem~\ref{SRH}. Note that if the entropy of $\chi$ is $O(1)$ then the number of distinct coordinates of $\chi$ is $O(n/\log n)$, so Theorem~\ref{SRH} implies
\[
	|\hat{1_S}(\chi)| \leq e^{-H(\chi) n/2 + o(n)} n!/n^n.
\]
Thus for any $H\leq 10$ we have
\begin{align*}
	\sum_{\chi \colon 9H/10 < H(\chi)\leq H} |\hat{1_S}(\chi)|^3
	&\leq \left|\left\{\chi \colon H(\chi)\leq H \right\}\right| \max_{\chi \colon H(\chi)\geq 9H/10} |\hat{1_S}(\chi)|^3\\
	&\leq e^{-7 H n/20 + o(n)} n!^3/n^{3n}.
\end{align*}
Decomposing the range $[\eps, 10]$ into ranges of this type and bounding crudely, we obtain
\[
\sum_{\chi \colon \eps \leq H(\chi)\leq 10} |\hat{1_S}(\chi)|^3
	\leq O(\log(1/\eps)) e^{-7\eps n/20 + o(n)} n!^3/n^{3n}.
\]

In the low-entropy range $H\leq \eps$, one can easily verify that $\chi = (r_1,\dots,r_n)$ must repeat some coordinate $r\in\hat{G}$ at least $(1-\eps)n$ times. Thus by global shift-invariance of $\hat{1_S}$ we have
\begin{equation} \label{eqn:low-entropy-splitting}
  \sum_{\chi \colon H(\chi)\leq \eps} \hat{1_S}(\chi)^3 = n \sum_{m=0}^{\eps n}\ \sum_{\substack{m\text{-sparse }\chi \\ H(\chi)\leq \eps}} \hat{1_S}(\chi)^3.
\end{equation}
By combining Proposition~\ref{rudiLinfty} with Theorem~\ref{sparseval} we have that for any $m \le \eps n$,
\begin{align*}
  \sum_{m\text{-sparse }\chi } |\hat{1_S}(\chi)|^3
  &\leq \(\max_{m\text{-sparse }\chi} |\hat{1_S}(\chi)|\) \sum_{m\text{-sparse } \chi} |\hat{1_S}(\chi)|^2\\
	&\leq e^{O\left(m^{3/2}/n^{1/2} + m^{1/2}\right)} 2^{-m/2} \frac{n!^3}{n^{3n}}\\
	&\leq e^{O\left(\eps^{1/2}m + m^{1/2}\right)} 2^{-m/2} \frac{n!^3}{n^{3n}}.
\end{align*}
As long as $\eps$ is sufficiently small depending on the constant implicit in the $O(\eps^{1/2}m)$ term, this is negligible except when $m$ has size $O(1)$.  In this range, we can apply Proposition~\ref{majorarcs}.  We introduce a further parameter $M$ and split the sum (\ref{eqn:low-entropy-splitting}) into the ranges $1 \le m \le 2M$ and $2M < m \le \eps n$, to obtain
\begin{align*}
  \sum_{\chi \colon H(\chi)\leq \eps} \hat{1_S}(\chi)^3
  &= n \sum_{m=0}^M \frac{(-1)^m}{2^m m!} \frac{n!^3}{n^{3n}} + n O_M\(\frac1n \frac{n!^3}{n^{3n}}\)\\
  &\qquad + O\left(n \sum_{m=2M}^{\eps n} e^{O\left(\eps^{1/2}m + m^{1/2}\right)} 2^{-m/2} \frac{n!^3}{n^{3n}} \right) \ .
\end{align*}

Finally, by combining the three bounds we have just proved for each range $H \le \eps$, $\eps \le H \le 10$ and $H \ge 10$, we deduce
\begin{align*}
	\sum_{\chi\in\hat{G}^n} \hat{1_S}(\chi)^3
	&= n \sum_{m=0}^M \frac{(-1)^m}{2^m m!} \frac{n!^3}{n^{3n}} + n O_M\(\frac1n \frac{n!^3}{n^{3n}}\)\\
  &\qquad + O\left( n \sum_{m=2M}^{\eps n} e^{O\left(\eps^{1/2}m + m^{1/2}\right)} 2^{-m/2} \frac{n!^3}{n^{3n}} \right)\\
  &\qquad + O\left( \log(1/\eps) e^{-7\eps n/20 + o(n)} \frac{n!^3}{n^{3n}} \right) \\
  &\qquad + O\left(e^{-3.5n} \frac{n!^3}{n^{3n}}\right).
\end{align*}
Theorem~\ref{supercooltheorembetterthanfermatslast} now follows by choosing $\eps$ and $M$ appropriately.

\bibliography{references}

\begin{thebibliography}{CGKN99}

\bibitem[B{\'o}n15]{handbookenumerative}
M.~B{\'o}na.
\newblock {\em Handbook of Enumerative Combinatorics}.
\newblock Chapman and Hall/CRC, 2015.

\bibitem[CGKN99]{coopersimulations}
C.~Cooper, R.~Gilchrist, I.~N. Kovalenko, and D.~Novakovi\'{c}.
\newblock Deriving the number of ``good'' permutations, with applications to
  cryptography.
\newblock {\em Kibernet. Sistem. Anal.}, (5):10--16, 187, 1999.

\bibitem[CK96]{cooperkovalenko}
C.~Cooper and I.~N. Kovalenko.
\newblock The upper bound for the number of complete mappings.
\newblock {\em Theory Probab. Math. Statist.}, 53:77--83, 1996.

\bibitem[Coo00]{cooperlower}
C.~Cooper.
\newblock A lower bound for the number of good permutations.
\newblock {\em Data Recording, Storage and Processing (Nat. Acad. Sci.
  Ukraine)}, 213:15--25, 2000.

\bibitem[CW10]{cavenaghwanless}
N.~J. Cavenagh and I.~M. Wanless.
\newblock On the number of transversals in {C}ayley tables of cyclic groups.
\newblock {\em Discrete Appl. Math.}, 158(2):136--146, 2010.

\bibitem[FS09]{flajolet-sedgewick}
P.~Flajolet and R.~Sedgewick.
\newblock {\em Analytic combinatorics}.
\newblock Cambridge University Press, Cambridge, 2009.

\bibitem[GL15]{glebovluria}
R.~Glebov and Z.~Luria.
\newblock On the maximum number of latin transversals.
\newblock 2015.
\newblock Preprint available at http://arxiv.org/pdf/1506.00983.pdf.

\bibitem[Kov96]{kovalenko}
I.~N. Kovalenko.
\newblock On an upper bound for the number of complete mappings.
\newblock {\em Kibernet. Sistem. Anal.}, (1):81--85, 188, 1996.

\bibitem[Kuz07]{kuznetsovone}
N.~Yu. Kuznetsov.
\newblock Applying fast simulation to find the number of good permutations.
\newblock {\em Cybernet. Systems Anal.}, 43:830--837, 2007.

\bibitem[Kuz08]{kuznetsovtwo}
N.~Yu. Kuznetsov.
\newblock Estimating the number of good permutations by a modified fast
  simulation method.
\newblock {\em Cybernet. Systems Anal.}, 44:547--554, 2008.

\bibitem[Kuz09]{kuznetsovthree}
N.~Yu. Kuznetsov.
\newblock Estimating the number of latin rectangles by the fast simulation
  method.
\newblock {\em Cybernet. Systems Anal.}, 45:69--75, 2009.

\bibitem[MMW06]{mckay}
B.~D. McKay, J.~C. McLeod, and I.~M. Wanless.
\newblock The number of transversals in a {L}atin square.
\newblock {\em Des. Codes Cryptogr.}, 40(3):269--284, 2006.

\bibitem[RVZ94]{rivinvardizimmermann}
I.~Rivin, I.~Vardi, and P.~Zimmermann.
\newblock The {$n$}-queens problem.
\newblock {\em Amer. Math. Monthly}, 101(7):629--639, 1994.

\bibitem[Slo]{oeis}
N.~J.~A. Sloane.
\newblock The on-line encyclopedia of integer sequences.
\newblock \url{https://oeis.org/}.

\bibitem[Tar15]{taranenko}
A.~A. Taranenko.
\newblock Multidimensional permanents and an upper bound on the number of
  transversals in {L}atin squares.
\newblock {\em J. Combin. Des.}, 23(7):305--320, 2015.

\bibitem[Var91]{vardi}
I.~Vardi.
\newblock {\em Computational recreations in {M}athematica}.
\newblock Addison-Wesley Publishing Company, Advanced Book Program, Redwood
  City, CA, 1991.

\bibitem[Wan11]{wanless}
I.~M. Wanless.
\newblock Transversals in {L}atin squares: a survey.
\newblock In {\em Surveys in combinatorics 2011}, volume 392 of {\em London
  Math. Soc. Lecture Note Ser.}, pages 403--437. Cambridge Univ. Press,
  Cambridge, 2011.

\end{thebibliography}
\bibliographystyle{alpha}
\end{document}